\documentclass[11pt]{article}
\usepackage{amsmath,amssymb,amsthm}
\usepackage{enumitem}
\usepackage[affil-it]{authblk}
\usepackage{a4wide}
\usepackage[dvipsnames]{xcolor}
\usepackage{hyperref}
\usepackage{xfrac}
\usepackage{tikz}
\usepackage{caption} 
\usetikzlibrary{arrows}
\usepackage{authblk}
\usepackage{graphicx}
\usepackage{mathtools}
\usepackage{rotating}

\newtheorem{thm}[equation]{Theorem}
\newtheorem{cl}[equation]{Claim}
\newtheorem{que}[equation]{Question}

\newtheorem{prop}[equation]{Proposition}
\newtheorem{lemma}[equation]{Lemma}
\theoremstyle{definition}
\newtheorem{defn}[equation]{Definition}
\newtheorem{remark}[equation]{Remark}

\numberwithin{equation}{section}
\newtheorem{post}[equation]{Postulate}

\newcommand\degg{\, \mathsf{deg} \,}
\newcommand\spann{\, \mathsf{span}}

\title{A route to quantum computing through the theory of quantum graphs}

\author[1,2,*]{Farrokh Razavinia\thanks{Supported by the Azarbaijan Shahid Madani University under grant No. 117.d.22844 - 08.07.2023} \thanks{This research was in part also supported by a grant from IPM (No. 1403170014)}
}
\affil[1]{Department of Mathematics, Azarbaijan Shahid Madani University, Tabriz 5375171379, Iran}
\affil[2]{School of Mathematics, Institute for Research in Fundamental Sciences (IPM), P.O. Box: 19395‐5746, Tehran, Iran}
\affil[]{f.razavinia@phystech.edu}

\date{}

\begin{document}

\maketitle

\begin{abstract}
Based on our previous works, and in order to relate them with the theory of quantum graphs and the quantum computing principles, we once again try to introduce some newly developed technical structures just by relying on our toy example, the coordinate ring of $n\times n$ quantum matrix algebra $M_q(n)$, and the associated directed locally finite graphs $\mathcal{G}(\Pi_n)$, and the Cuntz-Krieger $C^*$-graph algebras. Meaningly, we introduce a $(4i-6)$-qubit quantum system by using the Cuntz-Krieger $\mathcal{G}(\Pi_i)$-families associated to the $4i-6$ distinct Hamiltonian paths of $\mathcal{G}(\Pi_i)$, for $i\in\{2,\cdots,n\}$. We also will present a proof of a claim raised in our previous paper concerning the graph $C^*$-algebra structure and the associated Cuntz-Krieger $\mathcal{G}(\Pi_n)$-families. 
\newline\newline
\textbf{MSC Numbers (2020)}: Primary 46L05, 68Q09; Secondary 46L67, 81P40, 46L55, 17B81, 81P45.
 \hfill \newline
\textbf{Keywords}: quantum computing; quantum information theory; quantum state; quantum system; qubit; quantum graph; Cuntz-Krieger $C^*$-graph algebra; quantum permutation group; Hamiltonian path.
\end{abstract}

\section{An invitation to quantum graphs}\label{Se:1:}
Non-commutative Geometry is all about finding quantum (i.e. non-commutative) generalizations of some familiar structures, such as quantum topology, quantum probability (free probability), quantum groups, quantum permutation groups, and quantum information theory, etc., which triggers the following question:
\begin{que}\label{Que:1:.:}
    What is the quantum analogue of a graph?
\end{que}
This problem has been studied from many perspectives, but the notion of quantum graphs (also called non-commutative graphs in \cite{RVW12}) first has been coined and introduced by Duan, Severini, Winter \cite{RVW12} in terms of operator systems as the confusability graph of a quantum channel in quantum information theory. As an analogue of the fact that simple undirected classical graphs are irreflexive symmetric relations, Weaver \cite{W21} formulated quantum graphs as reflexive symmetric quantum relations on a von Neumann algebra, which somehow extends the results obtained in \cite{RVW12}, and following that quantum relations were introduced by Kuperberg, Weaver \cite{KW12}. Later on, following these works, Musto, Reutter, Verdon \cite{MRV18} formulated finite quantum graphs as the adjacency operators on tracial finite quantum sets, and Brannan et al. In \cite{BCEHPSW19} a generalization has been introduced for non-tracial settings. 

Set $\Gamma=(\Gamma^0,\Gamma^1)$ be a (directed) graph with vertex set $\Gamma^0=\{1,\cdots,n\}$, and the edge set $\Gamma^1=\{e_{ij}=(i,j)\mid \ \text{if}\ i\sim j\}$, where $\sim$ has the usual meaning that the vertices $i$ and $j$ are connected by edge $e_{ij}$. Then there are three different approaches (as we know) to the theory of quantum graphs, and they could be indicated as follows:
\begin{enumerate}[label=\Roman*)]
    \item \label{First} The First approach will be applied by quantizing the confusability graph of the classical channels. This could be done by using the matrix quantum graphs and the operator systems, by projecting $P_{\mathfrak{S}}$ onto the operator system $\mathfrak{S}$.\cite{RVW12}
    \item \label{Second}The second approach will be applied by quantizing the edge set $\Gamma^1$, by using the quantum relations and the projection $P_{\Gamma^1}$ from $\chi_{\Gamma^1}$.\cite{W21}
    \item \label{Third} The third approach will be applied by quantizing the adjacency matrix, by using the categorical theory of the quantum sets and the quantum functions, and projecting $P_{\Gamma}$ by using the adjacency matrix $A_{\Gamma}$.\cite{MRV18}
\end{enumerate}
It has been proved that under some appropriate identifications, the range of the above projections will be the same operator system!

In this paper, we will apply the first approach (\ref{First} to the quantum graphs by using the quantum confusability graph associated with the quantum channel applicable to our initial example $\mathcal{G}(\Pi_2)$ based on our toy example.

\section{Notation}
Throughout this paper, $\mathbb K$ will stand for the ground field, and it will be assumed to be arbitrary unless otherwise stated.

We often will simply denote the matrix algebra $M_d(\mathbb C)$ of $d\times d$ complex-valued matrices, with $M_d$, and the identity $d\times d$ matrix with $I_d$, and $M_q(n)$  will stand for the space of $n\times n$ quantum matrices, and by $\mathbb K[M_q(n)]$ we mean the ring of coordinate ring of $M_q(n)$.

$\mathcal{G}_n$ will stand for the directed locally finite connected graph associated with the defining relations of $\mathbb K[M_q(n)]$, and $\Pi_n$    represents the adjacency matrix of $\mathcal{G}_n$. After that, we will consider $\mathcal{G}_n:=\mathcal{G}(\Pi_n)$.

By $u=(u_{ij})_{i,j}$, we mean the magic unitary matrix and $\pi_n$ will stand for the unique commuting magic unitary matrix with $\Pi_n$.

For a Hilbert space $\mathcal{H}$, we will consider the set of all (bounded) linear operators from $\mathcal{H}$ to $\mathbb K$, with $B(\mathcal{H},\mathbb{K})$, and we write $B(\mathcal{H})\equiv B(\mathcal{H},\mathcal{H})$ for the set of all linear operators acting on $\mathcal{H}$.

We use $*$ to present the complex conjugate transpose of vectors and matrices. 

In some places, we will use the Dirac notation from quantum mechanics, meaning that a column vector which is called a ``ket'' and is written as $|\psi\rangle\in\mathbb C^d$, and $\langle\psi|\equiv|\psi\rangle^*$ to indicate its dual.

By $\{0,1\}^q$, we mean the binary strings of length $q$.

\section{Quantum Graphs, $C^*$-algebraic approach}

Even though the various approaches to quantum graphs, mentioned at the end of section \ref{Se:1:}, look quite different (meaning that they are from different perspectives. For instance they are from the $C^*$-algebraic perspective, operator spaces, and from the Frobenius algebras point of view), it is known that they ultimately will lead to the same non-commutative theory. 

In this paper, we will focus on a $C^*$-algebraic description of the theory of quantum graphs based on \cite{MRV18}.

To start, let $\Gamma=(\Gamma^0,\Gamma^1)$ be as before, for $\Gamma^0$ the set of vertices and $\Gamma^1$ the set of edges. We have the following definition, quoted from \cite{RVW12}.
\begin{defn}
    The quantum (non-commutative) graph associated with $\Gamma$ will be defined as the operator system 
    $$\mathfrak{S}_{\Gamma}:=\spann\{E_{ij}: (i,j)\in\Gamma^1 \ \text{or} \ i=j, \forall i,j\in\Gamma^0\}\subseteq M_n,$$
    for $E_{ij}$ the matrix units in $M_n$.
\end{defn}
More generally we have the following definition.
\begin{defn}\cite{MRV18}
    In general, any such operator system in $M_n$ will be called a matrix quantum graph.
\end{defn}
Next, we need another important concept which is called a quantum channel, and it is a completely positive and trace-preserving linear map defined as follows.
\begin{defn}\label{Defn:QC:}\cite{RVW12}
    In quantum information theory,  a quantum channel is represented by a completely positive trace-preserving map $\Psi:M_{i}\to M_{j}$, written in the Choi-Kraus form
    \begin{equation}
\Psi(\mathcal{X})=\sum_{m}S_m\mathcal{X}S_{m}^{*} \qquad \ \text{with} \ \qquad \sum_{m}S_{m}^{*}S_m=I_i,
    \end{equation}
     for every $\mathcal{X}\in M_i$, where $I_i$ is the identity matrix in $M_i$.
\end{defn}
As has been pointed out in section \ref{Se:1:}, there are three known approaches to the theory of quantum graphs (as we know), and the one which we are going to follow, is related to the confusability graph, and of course, the noncommutative version. Hopefully, here we are not required to quantize the classical version, as ours is already quantized!

Let us quote the following theorem from \cite{RVW12}.
\begin{thm}
    For a quantum channel $\Psi:M_n\to M_m$, and any $\mathcal{X}\in M_n$, with $\Psi(\mathcal{X})=\sum_{i=1}^{r}S_i\mathcal{X}S_{i}^{*}$, the confusability graph of $\Psi$ is the operator system
    \begin{equation}
        \mathcal{S}_{\Psi}=\spann\{S_{i}^{*}S_i: 1\leq i,j \leq r\}\subseteq M_n.
    \end{equation}
\end{thm}
%The following theorem provides us with some essential information concerning the structures of the quantum channels as transformations of quantum systems. 

It is known that quantum channels generalize the unitary evolution of the isolated quantum systems to the open quantum systems, on which by looking at them from the mathematical point of view, a quantum channel can be described as a linear completely positive trace-preserving map $\Psi$ from $M_d$ to itself, and in this paper, by a quantum channel we mean these kinds of maps!

Note that the trace preservation condition is almost essential, since quantum channels are usually meant to map density matrices to the density matrices, and the complete positivity condition could be stated in a way that for any $m\geq 1$, $\Psi\otimes I_d: M_{dm}(\mathbb C)\to M_{dm}(\mathbb C)$ is a positive map. 

The following equivalent three characterizations of the quantum channel $\Psi$ from $M_d$ to $M_d$ are known and have been quoted from \cite{RVW12}:
\begin{enumerate}[label=\roman*)]
    \item\label{St:1} Stinespring dilation. For every $\mathcal{X}\in M_m$, there exists a finite dimensional Hilbert space $\mathcal{H}$, of dimension $n=m^2$, and an isometry $S\in M_m$ acting on $\mathcal{H}$, i.e. $S:\mathbb C^m\to \mathbb C^m\otimes \mathbb C^n$, such that
    \begin{equation}
        \Psi(\mathcal{X})=\left(I_m\otimes Tr_n\right)(S\mathcal{X}S^*).
    \end{equation}
    \item\label{CK:2} Choi-Kraus decomposition. For every $\mathcal{X}\in M_m$, there exists an integer $k$ and isometries $S_1,\cdots,S_k\in M_m$ satisfying $\sum_{i}S_{i}^{*}S_i=I_m$, such that
    \begin{equation}
        \Psi(\mathcal{X})=\sum_{i=1}^{n}S_i\mathcal{X}S_{i}^{*}.
    \end{equation}
    \item\label{C:3} Choi matrix. Set $n=m^2$. The positive semidefinite Matrix $C_{\Psi}\in M_{n}$, satisfying $\left(I_m\otimes Tr_n\right)(C_{\Psi})=I_m$, is called the Choi matrix of $\Psi$ such that
    \begin{equation}
        C_{\Psi}:=\sum_{i,j=1}^{d}E_{ij}\otimes\Psi(E_{ij})\in M_m\otimes M_m.
    \end{equation}
\end{enumerate}
    
In this paper, we will follow the second characteristic (\ref{CK:2}, which has been already proven that is equivalent to other characteristics (\ref{St:1}, (\ref{C:3}.\cite{RVW12}

\section{$C^*$-Graph Algebras, a rout to quantum ($C^*$-algebraic) graphs}
By following the constructions in \cite{RH242}, and considering the finite locally connected (directed) graph $\mathcal{G}_2:=\mathcal{G}(\Pi_2)=(\mathcal{G}_{2}^{0},\mathcal{G}_{2}^{1})$, associated with the coordinate ring of the $2\times 2$-quantum matrix algebra $M_q(2)$, let us consider its set of vertices and edges, exactly as in \cite{RH242}, such that  $\mathcal{G}_{2}^{0}=\{x_{11}:=u, x_{12}:=v, x_{22}:=k, x_{21}:=w\}$ and $\mathcal{G}_{1}^{1}=\{x_{11}\overrightarrow{\sim}x_{12}:=e, x_{11}\overrightarrow{\sim}x_{21}:=f, x_{12}\overrightarrow{\sim}x_{22}:=h, x_{21}\overrightarrow{\sim}x_{22}:=g, x_{12}\overrightarrow{\sim}x_{21}:=i, x_{21}\overrightarrow{\sim}x_{12}:=j\}$, illustrated as in figure \ref{Fig:1:}.
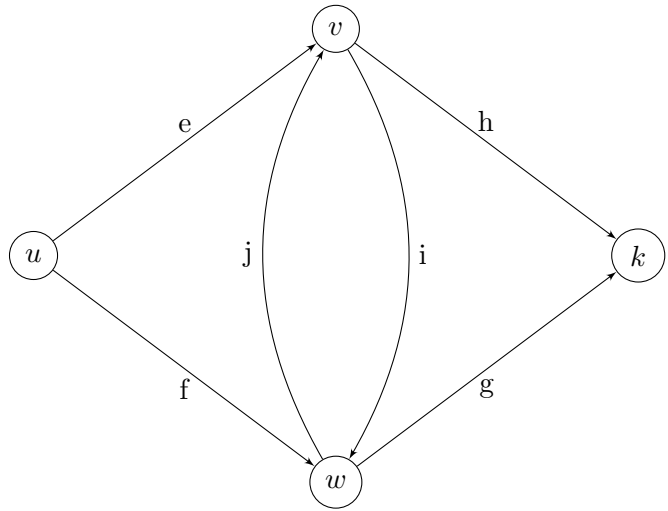
\vspace*{0.3cm}

\hspace*{2.7cm} \begin{tikzpicture}\label{Gra:2}
%[
%box/.style={draw,rectangle,minimum size=2cm,text width=1.5cm,align=left}]
\tikzset{vertex/.style = {shape=circle,draw,minimum size=0.7em}}
\tikzset{edge/.style = {->,> = latex'}}
% vertices
\node[vertex] (a) at  (0,0) {$u$};
\node[vertex] (b) at  (4,3) {$v$};
\node[vertex] (c) at  (8,0) {$k$};
\node[vertex] (d) at  (4,-3) {$w$};
\draw[edge] (a) to node[above] {e} (b);
\draw[edge] (b) to node[above] {h} (c);
\draw[edge] (a) to node[below] {f} (d);
\draw[edge] (d) to node[below] {g} (c);
\draw[edge] (b) to[bend left] node[right] {i} (d);
\draw[edge] (d) to[bend left] node[left] {j} (b);
\end{tikzpicture}
\captionof{figure}{\textbf{Directed locally connected graph related to $\Pi_2$}}\label{Fig:1:}

\vspace*{0.3cm}

In \cite{RH242}, we proved that there is an infinite-dimensional $C^*$-graph algebra associated with $\mathcal{G}(\Pi_2)$ as follows.
\begin{prop}\label{Prop:CKQ:1}
 By considering the underlying infinite dimensional Hilbert space $\mathcal{H}:=\ell^2(\mathbb N)$, the set 
\begin{align}
    S=\{& S_e:=\sum_{n=1}^{\infty}E_{6n,3n-2}, S_f:=\sum_{n=1}^{\infty}E_{6n-4,3n-2}, S_h:=\sum_{n=1}^{\infty}E_{6n-3,3n},\notag \\&S_g:=\sum_{n=1}^{\infty}E_{6n-4,3n-1}, S_i:=\sum_{n=1}^{\infty}E_{6n-1,3n}, S_j:=\sum_{n=1}^{\infty}E_{6n-3,3n-1}\},\label{Equ:PIs}
\end{align}
will define a Cuntz-Krieger $\mathcal{G}(\Pi_2)$-family and will give us an infinite dimensional graph $C^*$-algebra structure $\mathcal{C}^*(\Pi_2)$. 
\end{prop}
Later on, we will use the $\mathcal{G}(\Pi_2)$-Cuntz-Krieger family $S$, introduced in Proposition \ref{Prop:CKQ:1}, in order to define the associated quantum channel $\Psi_{\mathcal{G}(\Pi_2)}$ and the noncommutative confusability graph $\mathcal{S}_{\Psi_{\mathcal{G}(\Pi_2)}}$, and finally we will propose the associated 2-qubit quantum system, and the higher orders.

%Let $\mathcal{B}=C(\mathcal{G}^0)=C^* \left((P_v)_{v\in \mathcal{G}^0}\mid P_{v}^{*}=P_{v}^{2}=P_v, \sum_{v}P_v=1\right)$ be the $C^*$-algebra of functions on $\mathcal{G}^0$. 
\subsection{A route to quantum computing through the theory of quantum graphs}

The field of quantum error correction is concerned with protecting fragile quantum information from the unexpected errors which might happen while processing large-scale computations, in an effort to build a much more efficient quantum computer than their predecessors. In February 2019, Adrian Chapman while looking for some new approaches in order to tackle a long-standing search for the Holy Grail of quantum error correction, noticed a critical relation between the theory of quantum graphs (from the $C^*$-algebraic point of view) and the previous findings concerning the quantum error corrections. He tried to develop a new way of encoding quantum information that is resistant to errors by constructions and doesn't require active correction. What we do, is also somehow the same on an epsilon scale of what has been done, on which might be helpful, and could be simply stated as follows!

\hspace*{0.2cm} Following the above-stated statements, and in order to relate the core base of our work with the theory of quantum computation and the error corrections, let us invite $\mathcal{G}(\Pi_2)$ back to the scene, and consider its set of Cuntz-Krieger family (\ref{Equ:PIs}). 

\hspace*{0.2cm}Consider figure \ref{Fig:1:}, and note that $\mathcal{G}(\Pi_2)$ consists of two distinct  Hamiltonian paths. Let us call them $\mathcal{P}_1$ and $\mathcal{P}_2$, such that $\mathcal{P}_1$ starts from the vertex $u$ and passes through the edges $e, i, g$ and terminates in the vertex $k$. On the other hand, $\mathcal{P}_2$ starts from the vertex $u$ and passes through the edges $f, j, h$ and terminates in the vertex $k$, as has been demonstrated.

\hspace*{0.2cm}Then by employing the defining relations (\ref{Equ:PIs}),

it will not be too difficult to see that $S_{e}^{*}S_{e}+S_{i}^{*}S_{i}+S_{g}^{*}S_{g}=I$, (and $S_{f}^{*}S_{f}+S_{j}^{*}S_{j}+S_{h}^{*}S_{h}=I$) are true statements within $\mathcal{P}_1$ and $\mathcal{P}_2$ respectively, and then by using the definition \ref{Defn:QC:}, we might be interested in defining a completely positive trace-preserving map $\Psi_{\mathcal{G}(\Pi_2)}^{1}:M_{m}\to M_{m}$ (and $\Psi_{\mathcal{G}(\Pi_2)}^{2}:M_{m}\to M_{m}$); where $m$ will be specified later; in a way that we have $\Psi_{\mathcal{G}(\Pi_2)}^{1}(\mathcal{X})=\sum_{m}S_m\mathcal{X}S_{m}^{*}$ (and $\Psi_{\mathcal{G}(\Pi_2)}^{2}(\mathcal{X})=\sum_{m}S_m\mathcal{X}S_{m}^{*}$) for all $\mathcal{X}\in M_m$, and we will call it our quantum channel. After that, for $\ell\in\{e,g,i\}$ (and $\ell\in\{f,h,j\}$), the generating partial (matrix) isometries $S_{\ell}$, will be called the Choi-Kraus operators, and the noncommutative (confusability) graph of $\Psi_{\mathcal{G}(\Pi_2)}^{1}$ (and $\Psi_{\mathcal{G}(\Pi_2)}^{2}$) will be the operator system
    \begin{align}
        &\mathcal{S}_{\Psi_{\mathcal{G}(\Pi_2)}^{1}}=\spann\{S_{\ell}^{*}S_{\ell}\}\subseteq M_m,\\&\hspace{-0.3cm} \left(\text{and} \ \mathcal{S}_{\Psi_{\mathcal{G}(\Pi_2)}^{2}}=\spann\{S_{\ell}^{*}S_{\ell}\}\subseteq M_m,\right)
    \end{align}
in a way that completely characterizes the number of zero-error messages on which one can send through the quantum channel $\Psi_{\Pi_2}^{1} (\text{and}\  \Psi_{\Pi_2}^{2})$.
\begin{lemma}\label{Lem:QCH::}
    There are two distinct quantum channels associated with $\mathcal{G}(\Pi_2)$.
\end{lemma}
\begin{proof}
    As stated above, in $\mathcal{G}(\Pi_2)$ there are two distinct Hamiltonian paths $\mathcal{P}_1$ and $\mathcal{P}_2$, such that we have $S_{e}^{*}S_{e}+S_{i}^{*}S_{i}+S_{g}^{*}S_{g}=I$, and   $S_{f}^{*}S_{f}+S_{j}^{*}S_{j}+S_{h}^{*}S_{h}=I$, respectively, and it is known that associated to each of these relations, there exists quantum channels $\Psi_{\Pi_2}^{1}$ \text{and}  $\Psi_{\Pi_2}^{2})$, and the correspondent noncommutative (confusability) graphs $\mathcal{S}_{\Psi_{\mathcal{G}(\Pi_2)}^{1}}$ and $\mathcal{S}_{\Psi_{\mathcal{G}(\Pi_2)}^{2}}$.
\end{proof}
Let us try to have some estimations for the quantum channels from the Lemma \ref{Lem:QCH::}.

For any $\mathcal{X}\in M_4$, we have 

\begin{align*}
    \Psi_{\mathcal{G}(\Pi_2)}^{1}(\mathcal{X})&=\sum\limits_{m\in\{e,i,g\}}^{}S_m\mathcal{X}S_{m}^{*}\\&=S_e\mathcal{X}S_{e}^{*}+S_i\mathcal{X}S_{i}^{*}+S_g\mathcal{X}S_{g}^{*}\\&=\sum\limits_{n=1}^{\infty}E_{6n,3n-2}\mathcal{X}\sum\limits_{n=1}^{\infty}E_{3n-2,6n}+\sum\limits_{n=1}^{\infty}E_{6n-1,3n}\mathcal{X}\sum\limits_{n=1}^{\infty}E_{3n,6n-1} \\&\hspace*{0.3cm} +\sum\limits_{n=1}^{\infty}E_{6n-4,3n-1}\mathcal{X}\sum\limits_{n=1}^{\infty}E_{3n-1,6n-4}\\&=x_{22}\sum\limits_{n=1}^{\infty}E_{3n-1,6n-4}\\&=x_{22}(E_{2,2}+E_{5,8}+E_{8,14}+\cdots).
\end{align*}
And for $\Psi_{\mathcal{G}(\Pi_2)}^{2}$ we have
\begin{align*}
    \Psi_{\mathcal{G}(\Pi_2)}^{2}(\mathcal{X})&=\sum\limits_{m\in\{f,j,h\}}^{}S_m\mathcal{X}S_{m}^{*}\\&=S_f\mathcal{X}S_{f}^{*}+S_j\mathcal{X}S_{j}^{*}+S_h\mathcal{X}S_{h}^{*}\\&=\sum\limits_{n=1}^{\infty}E_{6n-4,3n-2}\mathcal{X}\sum\limits_{n=1}^{\infty}E_{3n-2,6n-4}+\sum\limits_{n=1}^{\infty}E_{6n-3,3n-1}\mathcal{X}\sum\limits_{n=1}^{\infty}E_{3n-1,6n-3} \\&\hspace*{0.3cm} +\sum\limits_{n=1}^{\infty}E_{6n-3,3n}\mathcal{X}\sum\limits_{n=1}^{\infty}E_{3n,6n-3}\\&=x_{22}\sum\limits_{n=1}^{\infty}E_{3n-2,6n-4}+x_{32}\sum\limits_{n=1}^{\infty}E_{3n-1,6n-3}+x_{33}\sum\limits_{n=1}^{\infty}E_{3n,6n-3}\\&=x_{22}(E_{1,2}+E_{4,8}+E_{7,14}+\cdots)+x_{32}(E_{2,3}+E_{5,9}+E_{8,15}+\cdots)\\&\hspace*{0.3cm}+x_{33}(E_{3,3}+E_{6,9}+E_{9,15}+\cdots).
\end{align*}
\begin{que}
    So, now the question is that, what do these estimations tell us?
\end{que}

\hspace*{0.2cm} Following the above findings, and in order to move a little bit further, and somehow alternate the direction to the theory of quantum computing, we use the relation between the degree of sink and source nodes and the number of distinct Hamiltonian paths in our directed graphs $\mathcal{G}(\Pi_n)$, in order to predict the accuracy of the possible quantum systems!

\hspace*{0.2cm}But before that, we note that the above statement, on the existence of a relation between the degree and the number of distinct Hamiltonian paths, is not always true! There are such relations, exactly when the source and the sink vertices have the same degree, and else there is no such relation.

\begin{lemma}\label{Lem::1:}
    The number of distinct Hamiltonian paths in $\mathcal{G}(\Pi_n)$ is equal to $4n-6$.
\end{lemma}
\begin{proof}
    We will prove this assertion by combining the recursive relation on $\mathcal{D}_n$ and the induction processes on $n$.

    To proceed with the proof, let us start by considering $\mathcal{G}(\Pi_2)$. Note that in $\mathcal{G}(\Pi_2)$ the degree is equal to $\mathcal{D}_2:=\mathcal{D}_{\mathcal{G}(\Pi_2)}=2=2(2-1)$. And as has been stated before, there are exactly $2=2(2-1)=2(\mathcal{D}_2-1)$ distinct Hamiltonian paths, namely $\mathcal{P}_1$ and $\mathcal{P}_2$. 
    
    However, in $G(\Pi_3)$, the degree is $\mathcal{D}_3:=\mathcal{D}_{\mathcal{G}(\Pi_3)}=4$, and once again, by doing some precaution, it is easy to see that $\mathcal{D}_3=2(3-1)$. And in accordance, it is easy to see that it possesses $6=2(4-1)=2(\mathcal{D}_3-1)$ distinct Hamiltonian paths. This algorithm continues until $G(\Pi_n)$, but let us prove it. We have
    \begin{align*}
        &\mathcal{D}_2=2=2(2-1),\\&
        \mathcal{D}_3=4=2(3-1)=2(2-1)+2,\\&
        \mathcal{D}_4=6=2(4-1)=2(2-1)+2+2\\&
        \mathcal{D}_5=8=2(5-1)=2(2-1)+2+2+2,
    \end{align*}
    and hence, we obtain the recursive relation $\mathcal{D}_i=\mathcal{D}_2+\mathcal{D}_{i-1}$, for any $i\in\{2,\cdots,n\}$. 

    Suppose the first step, i.e. $i=2$, and the $n$th step, i.e. $\mathcal{D}_n$ satisfies. We have 
    \begin{align*}
        \mathcal{D}_{n+1}=\mathcal{D}_2+\mathcal{D}_n&=2(2-1)+2(n-1)\\&=2(2)-2+2(n)-2\\&=2n\\&=2(n+1-1).
    \end{align*}
        
    And since, by definition, we have the number of distinct Hamiltonian paths $\# \mathcal{H}_{\mathcal{G}(\Pi_n)}=2(\mathcal{D}_n-1)$, and hence we have $\# \mathcal{H}_{\mathcal{G}(\Pi_n)}=2(2(n-1)-1)=4n-6$.
    
\end{proof}
\begin{remark}\label{Rem:::2::}
        Hence, from lemma \ref{Lem::1:}, we know that in $\mathcal{G}(\Pi_n)$ there are $4n-6=2(2n-3)$ distinct Hamiltonian paths with edge degree $n^2-1$. But, later on we will notice that we cannot use this as the dimension of our subsystems, but the dimension is very close to this number!
    \end{remark}

And we have the following result.
\begin{lemma}
    There are $4n-6$ distinct quantum channels associated with $\mathcal{G}(\Pi_n)$.
\end{lemma}
\begin{proof}
    The proof of this lemma, is almost clear, since by lemma \ref{Lem::1:}, we have that the number of distinct Hamiltonian paths in $\mathcal{G}(\Pi_n)$ is equal to $4n-6$, and we know that to each Hamiltonian path, we can associate a quantum channel and a correspondent noncommutative confusability graph.
\end{proof}

\hspace*{0.2cm}Then by employing these paths and the partial (matrix) isometries associated with each edge substituted in these Hamiltonian paths, we can formulate the associated quantum channel and the noncommutative (confusability) graph of $\mathcal{G}(\Pi_n)$, and after that, by associating to the vertices involved in each Hamiltonian path we might get some results concerning the $4n-6$-qubit quantum system, which is the goal!

\hspace*{0.2cm} The idea of using (distinct) Hamiltonian paths, came from the famous travelling salesman problem (\textit{TSP}) and its generalization, i.e. the travelling salesman path problem (\textit{TSPP}), and trying to find an optimized algorithm in order to solve that. Since the matrix representation of $\mathbb K[M_q(n)]$ possesses a very nice triangulated partitioning, meaning that the associated directed graphs $\mathcal{G}(\Pi_n)$ are triangulated, and as it is known there is a very common way in studying the \textit{TSP} by using polygonometric (polyhedral) methods, and this has triggered us to propose the following way in approaching these kinds of problems! 

\subsubsection{2-qubit quantum system}
Qubit is the basic unit in the quantum information theory, and in order to formally define it, we require the following first initial postulate of quantum mechanics.
\begin{post}[State Space Postulate]\cite{ALS22}
    In any isolated (physical) system, there is a complex space together with an inner product known as the state space of the system.
\end{post}
\begin{remark}
    Any physical system will completely be described by its unit vector from its state space, on which the vector will be called the state vector of the system.
\end{remark}

Depending on the underlying (directed) graph of the quantum system, the state space will be different. But let us just consider $\mathbb C^p$ to be the state space of our pre-assumed quantum system. Note that, depending on the choice, the state vectors of the state space, will completely describe the space. By a state vector, we mean the unit vector in $\mathbb C^p$, which will be called a qubit, in the language of the information theory, and we have the following definition.
\begin{defn}\label{Defn:Qu:}
    A linear combination $\phi=c_1|0\rangle+c_2|1\rangle$ of the basis vectors of $\mathbb C^2$, for $c_1,c_2\in\mathbb C$ will be called a qubit.
\end{defn}
\begin{remark}
    In Definition \ref{Defn:Qu:}, $c_i$s, for $i\in\{1,2\}$, usually (in the literature) will be called amplitudes, and they satisfy in the condition $|c_1|^2+|c_2|^2=1$, which is the usual normalization condition, it is not too difficult to see that the normalization condition will simply imply $\| |\phi\rangle\|=\sqrt{\langle\phi|\phi\rangle}=1$.
\end{remark}
At this point, we need the second postulate, which could be stated as follows and is concerned with qubit multiples.
\begin{post}\cite{ALS22}
If a system $\phi$ is a combination of two different (physical) systems $\phi_1$ and $\phi_2$, with the corresponding states $\mathbb E_1$ and $\mathbb E_2$, then the state of  $\phi$ will be $\mathbb E_1\otimes\mathbb E_2$, correspondent to $\phi_1\otimes\phi_2$.
\end{post}
\begin{remark}
    In the literature, mostly $\otimes$ symbol is omitted and $|\phi_1\rangle|\phi_2\rangle$ or $|\phi_1\phi_2\rangle$ is written in place of the composition of states $|\phi_1\rangle$ and $|\phi_2\rangle$. But, here in this paper we will not follow this notation, and we will refer to it as just the usual composition.
\end{remark}
We have the following important definition, which has been quoted from \cite{ALS22}.
\begin{defn}
      The state of a $q$-qubit quantum system is a unit vector in $(\mathbb C^2)^{\otimes q}$.
\end{defn}
Now, getting back to the main constructions, let $\{0,1\}^q$ be the binary strings of length $q$. Then by using the binary representation of the natural number $n=2^{k-1}a_{k-1}+\cdots+2a_1+a_0$, for $a_j\in\{0,1\}$, for all $0\leq j\leq k-1$, represented as $s=a_{k-1}a_{k-2}\cdots a_1a_0$, it is known that we can define a $q$-qubit system as \cite{ALS22}
\begin{align*}
    &|\phi\rangle=\sum_{j\in\{0,1\}^q}\alpha_j|j\rangle,\\&\text{where} \ \sum_{j\in\{0,1\}^q}|\alpha_j|^2=1,
\end{align*}
for $\alpha_j\in\mathbb C$. Equivalently we have
\begin{align*}
    &|\phi\rangle=\sum_{k=0}^{2^q-1}\alpha_k|k\rangle,\\&\text{where} \ \sum_{k=0}^{2^q-1}|\alpha_k|^2=1,
\end{align*}
for $\alpha_k\in\mathbb C$.

At this point, we already have what we need in order to enter quantum computing and quantum information theory!

From remark \ref{Rem:::2::}, for $i\in\{2,\cdots,n\}$, in $\mathcal{G}(\Pi_i)$, we know that the dimension of our subsystems are not equal to $i^2-1$, which is the edge dimension of the Hamiltonian paths. Instead, let us take $k=i^2-1-((i-1)^2+1)=2i-3$. Now consider $(\mathbb C^2)^{\otimes 2i-3}$ to be our subsystems of dimension $2^{4i-6}$ divided to the product of $2\times(2i-3)$ 1-qubit states.

Now, let us invite $\mathcal{G}(\Pi_2)$ back to the scene. Take $m=4$ in quantum channel $\Psi_{\mathcal{G}(\Pi_2)}$.

    Here, our quantum system will be a combination of two subsystems $\mathbb E_1$ and $\mathbb E_2$ with the correspondent states $\phi_1$ and $\phi_2$, and since in this case, the Hamiltonian paths are a combination of three edges, hence let $(\mathbb C^2)^{\otimes 1}\cong\mathbb C^2$ be our state space $\phi$. Consider $\phi_1=\alpha_1|0\rangle +\alpha_2|1\rangle$ and $\phi_2=\beta_1|0\rangle +\beta_2|1\rangle$ such that we have $\sum_{i=1}^{2}\alpha_{i}^{2}=\sum_{i=1}^{2}\beta_{i}^{2}=1$. Not that the only possible options for $\alpha_i$s and $\beta_i$s are $\pm\frac{1}{\sqrt{2}}$ and $\pm\frac{i}{\sqrt{2}}$.

Now we can define a 2-qubit system associated with the pre-assumed quantum system as follows
\begin{align*}
    &|\phi\rangle=\sum_{j\in\{0,1\}^2}\alpha_j|j\rangle=\alpha_{00}|00\rangle+\alpha_{01}|01\rangle+\alpha_{10}|10\rangle+\alpha_{11}|11\rangle\\&\hspace*{0.5cm}=\alpha_{0}|0\rangle+\alpha_{1}|1\rangle+\alpha_{2}|2\rangle+\alpha_{3}|3\rangle,\\&\text{where} \ \sum_{j\in\{0,1\}^2}|\alpha_j|^2=1,
\end{align*}
or, equivalently
\begin{align}
&|\phi\rangle=\sum_{k=0}^{3}\alpha_k|k\rangle=\alpha_0|0\rangle+\alpha_1|1\rangle+\alpha_2|2\rangle+\alpha_3|3\rangle,\label{EQU:QS:2:}\\&\text{where} \ \sum_{k=0}^{3}|\alpha_k|^2=1.\notag
\end{align}
It is easy to see that, in (\ref{EQU:QS:2:}), the only possible options for $\alpha_k$s are $\pm\frac{1}{2}$ and $\pm\frac{i}{2}$, and as a result we have the following proposition.
\begin{prop}
    Our pre-assumed 2-qubit, and 6-qubit quantum systems associated with $\mathcal{G}(\Pi_2)$, and $\mathcal{G}(\Pi_3)$, respectively, are not entangled!
\end{prop}
\begin{proof}
    The proof of this proposition is almost a trivial conclusion of the Definition \ref{DEFN:QE::}, because, in any case, $|\phi\rangle$, and $|\psi\rangle$, respectively for $\mathcal{G}(\Pi_2)$ and $\mathcal{G}(\Pi_3)$, will always be equal to the tensor product $|\phi_1\rangle\otimes|\phi_2\rangle$, because in $\mathcal{G}(\Pi_2)$ the amplitudes $|\phi_1\rangle\otimes|\phi_2\rangle$ will be equal to $\pm\frac{1}{2}$ and $\pm\frac{i}{2}$, just the same as in $|\phi\rangle$, and in $\mathcal{G}(\Pi_3)$ the amplitudes $|\psi_1\rangle\otimes|\psi_2\rangle$ will be equal to $\pm\frac{1}{2^3}$ and $\pm\frac{i}{2^3}$, just the same as in $|\psi\rangle$
\end{proof}
\begin{defn}\label{DEFN:QE::}\cite{ALS22}
    A quantum state $|\phi\rangle\in(\mathbb C^2)^{\otimes q}$ is a product state if it can be expressed as a tensor product $|\phi_1\rangle\otimes\cdots\otimes|\phi_q\rangle$ of $q$ 1-qubit states. Otherwise, it is entangled.
\end{defn}
\subsubsection{($4n-6$)-qubit quantum system}
We have the following claim.

\begin{cl}
    In the case of $\mathcal{G}(\Pi_n)$, the pre-assumed $(4n-6)$-qubit quantum system is not entangled.
\end{cl}
Note that in order to prove or disprove this claim, the interested person first needs to prove or disprove Claim 2.7., from \cite{RH242}, on which we will try to present a sketch of its proof here in this paper as the following Theorem.

\begin{thm}\label{Cl:Op:1}
   For $\mathcal{G}(\Pi_n)=(\mathcal{G}^0,\mathcal{G}^1)$ the associated directed locally finite graphs with $\mathbb K[M_q(n)]$, and $\Pi_n$ the associated adjacency matrices, and $\mathcal{H}:=\ell^2(\mathbb N)$ the underlying infinite dimensional Hilbert space. The claim is that the set 

    $$S=\{ S_i:=\sum_{j=1}^{\infty}\prescript{i}{}{E}_{\mathcal{E}j-A,(n^2-1)j-D} \mid \ \text{for fixed} \ 1\leq i\leq \frac{(n^3+n^2)(n-1)}{2}\}\label{Equ:PIs},$$

is a Cuntz-Krieger $\mathcal{G}(\Pi_n)$-family for $D\in\{0,\cdots,n^2-2\}$, and $\mathcal{E}$ depends on the degree of the exit edges to the vertex $e_{hk}$, where $i$ is considered as an exit edge, i.e. if $\overset{\rightarrow}{\degg}_{hk}=2$, then we will have $\mathcal{E}=2(n^2-1)$, and if it is 3, then we will have $\mathcal{E}=3(n^2-1)$, and so on, and $A\in\{0,\cdots,\overset{\rightarrow}{\degg}_{hk}\times(n^2-1)\}$, and gives us a graph $C^*$-algebra structure $\mathcal{C}^*(\Pi_n)$.
\end{thm}
\begin{proof}
    In order to prove, we will proceed by induction. But before proceeding, let us clarify the almost approximate values of the parameters $D$. Starting with Hamiltonian paths, we know that they start from the only source vertex $e_{1,1}$ and terminate at the only sink vertex $e_{n,n}$. Hence depending on the first connection, we might have each of $(n^2-1)i-D$, for $D\in\{0,\cdots,n^2-2\}$. For example, if the first connection is $e_{1,1}\rightarrow e_{1,2}$, then we have $D=0$, or if it is $e_{1,1}\rightarrow e_{2,1}$, then we have $D=1$, and \textit{etc}.

    In \cite{RH242}, we already have proved that the assertion is true for $n=2$. Bellow, we will prove that it also is true for $n=3$.

    Exactly in a same approach as in \cite[Proposition 3.1]{RH242}, we note that there might be finite or infinite dimensional sets of projections. But in order to skip the complication, and to shorten the proof, we skip proving that for some certain vertices $e_{i,j}$, we should have $\dim(P_{e_{i,j}})>\infty$, and based on an already proven fact, all the other projections also should be infinite dimensional. So, let us try to construct an appropriate \textit{CK-}$\mathcal{G}(\Pi_3)$ family satisfying the relations $S_{e}^{*}S_e=P_{r(e)}$, for all edges $e\in\mathcal{G}(\Pi_3)^1$, and $P_{e_{ij}}=\sum\limits_{s(e)=e_{i,j}}S_eS_{e}^{*}$ for the case when $e_{i,j}\in\mathcal{G}(\Pi_3)^0$ is not a sink. 
    
    \hspace*{0.0cm} \begin{tikzpicture}
%[
%box/.style={draw,rectangle,minimum size=2cm,text width=1.5cm,align=left}]
\tikzset{vertex/.style = {shape=circle,draw,minimum size=0.2em}}
\tikzset{edge/.style = {->,> = latex'}}
% vertices
\node[vertex] (a) at  (-1,1) {$e_{21}$};
\node[vertex] (b) at  (4,3) {$e_{22}$};
\node[vertex] (c) at  (9,1) {$e_{23}$};
\node[vertex] (d) at  (4,-3) {$e_{31}$};
\node[vertex] (e) at  (-2,-4) {$e_{11}$};
\node[vertex] (f) at  (10,-4) {$e_{33}$};
\node[vertex] (g) at  (0.7,0.5) {$e_{12}$};
\node[vertex] (h) at  (2.3,0.5) {$e_{13}$};
\node[vertex] (i) at  (7.3,-1.2) {$e_{32}$};
%\node[vertex] (a1) at (1.5,0) {};
%\node[vertex] (a2) at (3,0) {};
%edges
%\draw[edge] (a) to node[above] {e} (b);
\draw[edge] (a) to node[above] {$j_5$} (b);
\draw[edge] (b) to node[above] {\hspace{-0.9cm}\begin{turn}{-20}$g_6=f_6$\end{turn}} (c);
\draw[edge] (a) to node[below] {\hspace{0.83cm}\begin{turn}{-220}$e_3=f_3$\end{turn}} (d);
\draw[edge] (e) to node[left] {\begin{turn}{-280}$g_1=h_1$\end{turn}} (a);
\draw[edge] (c) to node[right] {\hspace{-0.4cm} \begin{turn}{-80}$i_8=e_8=j_8=h_8$\end{turn}} (f);
\draw[edge] (d) to node[below] {$1$} (f);
\draw[edge] (e) to node[below] {$i_1$} (d);
\draw[edge] (g) to node[above] {\textcolor{red}{$g_3=h_3$}} (h);
\draw[edge] (h) to node[above] {$14$} (f);
\draw[edge] (g) to node[below] {$13$} (i);
\draw[edge] (h) to node[below] {$7$} (c);
\draw[edge] (b) to node[above] {$e_6=h_6=i_6=j_6$} (i);
\draw[edge] (e) to node[above] {\begin{turn}{-296}$e_1=f_1$\end{turn}} (g);
\draw[edge] (e) to node[below] {$j_1$} (h);
\draw[edge] (g) to node[above] {$i_5$} (b);
\draw[edge] (a) to node[below] {$12$} (c);
\draw[edge] (i) to node[above] {$g_8=f_8$} (f);
%\draw[edge] (c) to node[above] {$g_7=f_7$} (i);
%	\draw[edge] (a)  to[bend left] (a1);
%	\draw[edge] (a1) to[bend left] (a);

%	\draw[edge] (a1) to[bend left] (a2);
%	\draw[edge] (a2) to[bend left] (a1);

%	\path (a2) to node {\dots} (c);
%	\node [shape=circle,minimum size=1.5em] (a3) at (4.5,0) {};
%	\draw[edge] (a2) to[bend left] (a3);
%	\draw[edge] (a3) to[bend left] (a2);

%	\node [shape=circle,minimum size=0.7em] (c1) at (6.5,0) {};
\draw[edge] (b) to[bend left] node[left] {\textcolor{red}{$10$}} (d);
\draw[edge] (d) to[bend left] node[right] {$g_5=h_5$} (b);
\draw[edge] (h) to[bend left] node[left] {\textcolor{RoyalBlue}{\hspace{-3cm}\begin{turn}{-260}$g_4=h_4=j_2$\end{turn}}} (d);
\draw[edge] (d) to[bend left] node[right] {\hspace{0.5cm}\begin{turn}{-250}$e_4=f_4=i_2$\end{turn}} (h);
\draw[edge] (h) to[bend left] node[below] {$i_3$} (a);
\draw[edge] (a) to[bend left] node[above] {$8$} (h);
\draw[edge] (a) to[bend left] node[above] {\textcolor{RoyalBlue}{$g_2=h_2=i_4$}} (g);
\draw[edge] (g) to[bend left] node[below] {\textcolor{red}{$e_2=f_2=j_4$}} (a);
\draw[edge] (b) to[bend left] node[below] {\textcolor{red}{$9$}} (h);
\draw[edge] (h) to[bend left] node[right] {\textcolor{red}{\hspace{-0.5cm}\begin{turn}{50}$e_5=f_5$\end{turn}}} (b);
\draw[edge] (i) to[bend left] node[above] {\textcolor{Red}{\hspace{-1cm}\begin{turn}{0}$e_7=h_7=i_7=j_7$\end{turn}}} (c);
\draw[edge] (c) to[bend left] node[below] {\textcolor{Red}{\vspace{-1cm}\begin{turn}{-320}$g_7=f_7$\end{turn}}} (i);
\draw[edge] (h) to[bend left] node[below] {\textcolor{RoyalBlue}{$2$}} (i);
\draw[edge] (i) to[bend left] node[above] {\textcolor{RoyalBlue}{$3$}} (h);
\draw[edge] (c) to[bend left] node[below] {\textcolor{RoyalBlue}{$5$}} (d);
\draw[edge] (d) to[bend left] node[below] {\textcolor{RoyalBlue}{$4$}} (c);
\draw[edge] (d) to node[below] {\textcolor{Red}{$11$}} (i);
\draw[edge] (g) to[bend right] node[below] {\textcolor{Red}{$6$}} (d);
\draw[edge] (d) to[bend right] node[above] {\textcolor{Red}{$j_3$}} (g);
%\node [below=0.9cm, align=flush center,text width=8cm] at (d)
%{
%	$M_1$
%}

\end{tikzpicture}
\vspace*{0.2cm}

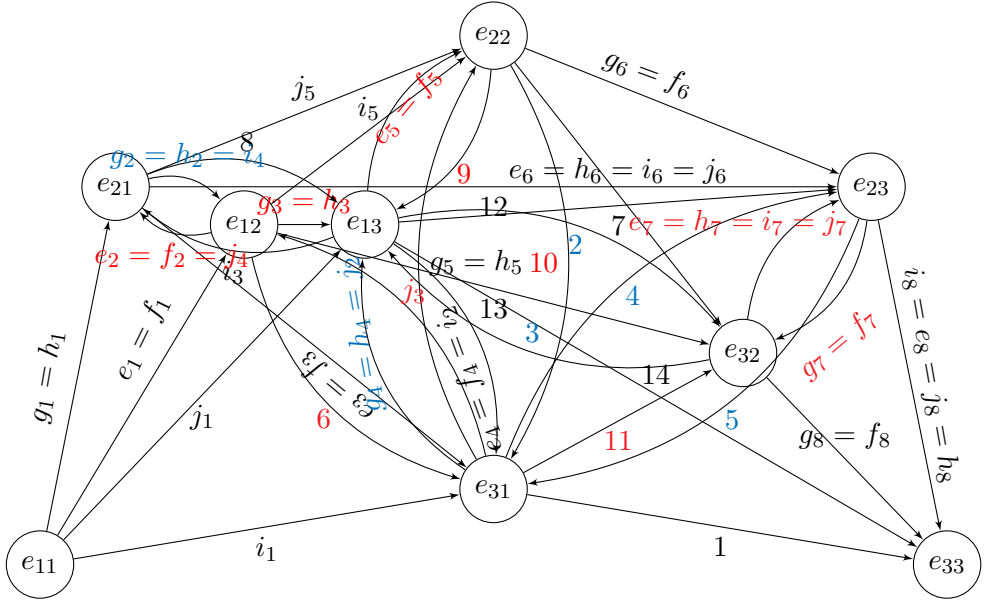
\captionof{figure}{\textbf{Directed locally connected graph $\mathcal{G}(\Pi_3)$}}

\vspace*{0.5cm}
    
 We note that $\mathcal{G}(\Pi_3)$ consists of $6$ Hamiltonian paths listed below:

 \begin{align*}
     &\mathcal{H}_1: e_{11}\overset{e_1}{\longrightarrow}e_{12}\overset{e_2}{\longrightarrow}e_{21}\overset{e_3}{\longrightarrow}e_{31}\overset{e_4}{\longrightarrow}e_{13}\overset{e_5}{\longrightarrow}e_{22}\overset{e_6}{\longrightarrow}e_{32}\overset{e_7}{\longrightarrow}e_{23}\overset{e_8}{\longrightarrow}e_{33},\\&\mathcal{H}_2:e_{11}\overset{f_1}{\longrightarrow}e_{12}\overset{f_2}{\longrightarrow}e_{21}\overset{f_3}{\longrightarrow}e_{31}\overset{f_4}{\longrightarrow}e_{13}\overset{f_5}{\longrightarrow}e_{22}\overset{f_6}{\longrightarrow}e_{23}\overset{f_7}{\longrightarrow}e_{32}\overset{f_8}{\longrightarrow}e_{33},\\&\mathcal{H}_3:e_{11}\overset{g_1}{\longrightarrow}e_{21}\overset{g_2}{\longrightarrow}e_{12}\overset{g_3}{\longrightarrow}e_{13}\overset{g_4}{\longrightarrow}e_{31}\overset{g_5}{\longrightarrow}e_{22}\overset{g_6}{\longrightarrow}e_{23}\overset{g_7}{\longrightarrow}e_{32}\overset{g_8}{\longrightarrow}e_{33},\\&\mathcal{H}_4:e_{11}\overset{h_1}{\longrightarrow}e_{21}\overset{h_2}{\longrightarrow}e_{12}\overset{h_3}{\longrightarrow}e_{13}\overset{h_4}{\longrightarrow}e_{31}\overset{h_5}{\longrightarrow}e_{22}\overset{h_6}{\longrightarrow}e_{32}\overset{h_7}{\longrightarrow}e_{23}\overset{h_8}{\longrightarrow}e_{33},\\&\mathcal{H}_5:e_{11}\overset{i_1}{\longrightarrow}e_{31}\overset{i_2}{\longrightarrow}e_{13}\overset{i_3}{\longrightarrow}e_{21}\overset{i_4}{\longrightarrow}e_{12}\overset{i_5}{\longrightarrow}e_{22}\overset{i_6}{\longrightarrow}e_{32}\overset{i_7}{\longrightarrow}e_{23}\overset{i_8}{\longrightarrow}e_{33},\\&\mathcal{H}_6:e_{11}\overset{j_1}{\longrightarrow}e_{13}\overset{j_2}{\longrightarrow}e_{31}\overset{j_3}{\longrightarrow}e_{12}\overset{j_4}{\longrightarrow}e_{21}\overset{j_5}{\longrightarrow}e_{22}\overset{j_6}{\longrightarrow}e_{32}\overset{j_7}{\longrightarrow}e_{23}\overset{j_8}{\longrightarrow}e_{33}.
 \end{align*}

    Now the claim is that the desired set of partial isometries will be as follows:

    \begin{align*}
        &S_{e_1}=S_{f_1}=E_{32n-31,8n}, S_{g_1}=S_{h_1}=E_{32n-23,8n-1}, S_{i_1}=E_{32n-7,8n-2},\\& S_{j_1}=E_{32n-15,8n-3},\\&S_{g_2}=S_{h_2}=S_{i_4}=E_{40n-25,8n}, S_{e_2}=S_{f_2}=S_{j_4}=E_{40n-32,8n-1}, S_{j_2}=E_{48n-35,8n-2},\\& S_{i_2}=E_{48n-42,8n-3}, \\&S_{j_3}=E_{48n-10,8n}, S_{i_3}=E_{48n-27,8n-1}, S_{e_3}=S_{f_3}=E_{40n-33,8n-2}, \\&S_{g_3}=S_{h_3}=E_{40n-24,8n-3}, \\&S_{i_4}=E_{40n-25,8n}, S_{j_4}=E_{40n-32,8n-1}, S_{g_4}=S_{h_4}=S_{j_2}=E_{48n-35,8n-2}, \\&S_{e_4}=S_{f_4}=S_{i_2}=E_{48n-42,8n-3},\\&S_{e_5}=S_{f_5}=E_{48n-43,8n-4}, S_{g_5}=S_{h_5}=E_{48n-34,8n-4}, S_{i_5}=E_{40n-8,8n-4}, \\&S_{j_5}=E_{40n-1,8n-4},\\&S_{e_6}=S_{h_6}=S_{i_6}=S_{j_6}=E_{32n-28,8n-5}, S_{g_6}=S_{f_6}=E_{24n-14,8n-5}, S_{i_6}=E_{32n-28,8n-5}, \\&S_{j_6}=E_{32n-28,8n-5},\\&S_{e_7}=S_{h_7}=S_{i_7}=S_{j_7}=E_{24n-21,8n-6}, S_{g_7}=S_{f_7}=E_{32n-20,8n-6}, S_{i_7}=E_{24n-21,8n-6}, \\&S_{j_7}=E_{24n-21,8n-6},\\&S_{e_8}=S_{h_8}=S_{i_8}=S_{j_8}=E_{24n-22,8n-7}, S_{g_8}=S_{f_8}=E_{24n-13,8n-7}, S_{i_8}=E_{24n-22,8n-7}, \\&S_{j_8}=E_{24n-22,8n-7},
    \end{align*}
    and the rest of the isometries are apart from the set of Hamiltonian paths, and could be outlined as follows:
    \begin{align*}
        &S_{1}=E_{48n-2,8n-7}, S_{14}=E_{48n-11,8n-7}, S_{4}=E_{48n-29,8n-7},\\& S_{7}=E_{48n-19,8n-6}, S_{12}=E_{40n-9,8n-6},\\& S_{2}=E_{48n-3,8n-5}, S_{11}=E_{48n-26,8n-5}, S_{13}=E_{40n-11,8n-5}, \\& S_{3}=E_{24n-5,8n-3}, S_{8}=E_{40n-17,8n-3}, S_{9}=E_{32n-12,8n-3}, \\& S_{5}=E_{24n-6,8n-2}, S_{6}=E_{40n-16,8n-2}, S_{5}=E_{32n-4,8n-2}.
    \end{align*}
    Now the goal is to prove that the above isometries satisfy in the relations of Cuntz-Krieger $\mathcal{G}(\Pi_3)$-family, i.e. $S_{e}^{*}S_e=P_{r(e)}$, for all edges $e\in\mathcal{G}(\Pi_3)^1$, and $P_{e_{ij}}=\sum\limits_{s(e)=e_{i,j}}S_eS_{e}^{*}$ for the case when $e_{i,j}\in\mathcal{G}(\Pi_3)^0$ is not a sink, and this could be verified simply. Here we just put it over as follows.

    \begin{align*}
        &P_{e_{13}}=S_{3}^{*}S_3=S_{g_3}^{*}S_{g_3}=S_{j_1}^{*}S_{j_1}=S_{e_4}^{*}S_{e_4}=S_{9}^{*}S_{9}=S_{8}^{*}S_{8}=S_{e_5}S_{e_5}^{*}+S_{g_4}S_{g_4}^{*}+S_{i_3}S_{i_3}^{*}\\& \hspace{0.7cm}+S_{7}S_{7}^{*}+S_{14}S_{14}^{*}+S_{2}S_{2}^{*},\\&P_{e_{22}}=S_{e_5}^{*}S_{e_5}=S_{g_5}^{*}S_{g_5}=S_{j_5}^{*}S_{j_5}=S_{i_5}^{*}S_{i_5}=S_{e_6}S_{e_6}^{*}+S_{g_6}S_{g_6}^{*}+S_{9}S_{9}^{*}+S_{10}S_{10}^{*},\\&P_{e_{23}}=S_{e_7}^{*}S_{e_7}=S_{12}^{*}S_{12}=S_{g_6}^{*}S_{g_6}=S_{7}^{*}S_{7}=S_{4}^{*}S_{4}=S_{e_8}S_{e_8}^{*}+S_{g_7}S_{g_7}^{*}+S_{5}S_{5}^{*}, \\&P_{e_{32}}=S_{2}^{*}S_{2}=S_{g_7}^{*}S_{g_7}=S_{e_6}^{*}S_{e_6}=S_{13}^{*}S_{13}=S_{11}^{*}S_{11}=S_{e_7}S_{e_7}^{*}+S_{g_8}S_{g_8}^{*}+S_{3}S_{3}^{*},
        \end{align*}
        
        \begin{align*}
   &P_{e_{33}}=S_{1}^{*}S_{1}=S_{14}^{*}S_{14}=S_{e_8}^{*}S_{e_8}=S_{g_8}^{*}S_{g_8},\\&P_{e_{31}}=S_{i_1}^{*}S_{i_1}=S_{e_3}^{*}S_{e_3}=S_{10}^{*}S_{10}=S_{g_4}^{*}S_{g_4}=S_{5}^{*}S_{5}=S_{6}^{*}S_{6}=S_{e_4}S_{e_4}^{*}+S_{g_5}S_{g_5}^{*}+S_{11}S_{11}^{*}\\&\hspace{0.7cm}+S_{4}S_{4}^{*}+S_{j_3}S_{j_3}^{*}+S_{1}S_{1}^{*},\\&P_{e_{11}}=S_{e_1}^{*}S_{e_1}+S_{g_1}S_{g_1}^{*}+S_{i_1}S_{i_1}^{*}+S_{j_1}S_{j_1}^{*},\\&P_{e_{21}}=S_{g_1}^{*}S_{g_1}=S_{e_2}^{*}S_{e_2}=S_{i_3}^{*}S_{i_3}=S_{e_3}S_{e_3}^{*}+S_{g_2}S_{g_2}^{*}+S_{8}S_{8}^{*}+S_{12}S_{12}^{*}+S_{j_5}S_{j_5}^{*},\\&P_{e_{12}}=S_{g_2}^{*}S_{g_2}=S_{e_1}^{*}S_{e_1}=S_{j_3}^{*}S_{j_3}=S_{e_2}S_{e_2}^{*}+S_{g_3}S_{g_3}^{*}+S_{6}S_{6}^{*}+S_{i_5}S_{i_5}^{*}+S_{13}S_{13}^{*},
    \end{align*}
and at this point, we are not going to have a detailed proof of the above relations, since it is easy to see them. This provide us with the desired Cuntz-Krieger $\mathcal{G}(\Pi_3)$-family, and shows that the $C^*(S,P)$ is an infinite-dimensional $C^*$-algebra.

Now in order to proceed further, let us see what we have. For $n=2$, we have $S_i=\sum_{j=1}^{\infty}\prescript{i}{}{E}_{\mathcal{E}j-A,3j-D}$, and for $n=3$ case we have $S_i=\sum_{j=1}^{\infty}\prescript{i}{}{E}_{\mathcal{E}j-A,8j-D}$, and $S_i=\sum_{j=1}^{\infty}\prescript{i}{}{E}_{\mathcal{E}j-A,15j-D}$ for the $n=4$ case, and so on. So, we get a sequence of numbers $3, 8, 15, 24, 35, \cdots$. In order to have a defining rule for this sequence, we use the recurrence relations $a_{n-2}+2n=a_{n-1}$, and $b_{n-2}+1=b_{n-1}$, for $n\in\{2,\cdots\}$, together with conditional relations $a_0=1$ and $b_0=2$. 

Now consider the recurrence relation $h_n+2n+1=h_{n+1}$, for $h_n=b_{n-2}+a_{n-2}$ and the conditional relation $h_2=b_0+a_0=3$.

Hence, we get the following relation
$$S=\{ S_i:=\sum_{j=1}^{\infty}\prescript{i}{}{E}_{\mathcal{E}j-A,h_nj-D} \mid \ \text{for fixed} \ 1\leq i\leq \frac{(n^3+n^2)(n-1)}{2}\}\label{Equ:PIs:::},$$
for $h_{n}$ as above,and $n\in\{2,\cdots\}$. Now we can proceed by induction on $n$. For $n=2,3$, we already have seen that the assertion holds. Suppose that for $\mathcal{G}(\Pi_i)$, $i\in\{2,3,\cdots\}$ (\ref{Equ:PIs:::}) satisfies.

Now the claim is that for $\mathcal{G}(\Pi_{n+1})$ we have
$$S=\{ S_i:=\sum_{j=1}^{\infty}\prescript{i}{}{E}_{\mathcal{E}j-A,h_{n+1}j-D} \mid \ \text{for fixed} \ 1\leq i\leq \frac{(n^3+n^2)(n-1)}{2}\}\label{Equ:PIs:::},$$
which is almost clear, since we have $(n+1)^2-1=n^2-1+2n+1=h_n+2n+1=h_{n+1},$ and we are done with the induction steps!

%    Now suppose that the assertion is true for $\mathcal{G}(\Pi_n)$. We will prove that it is true for $\mathcal{G}(\Pi_{n+1})$.

 %   $\mathcal{G}(\Pi_n)$ has $4n-6$ Hamiltonian paths with edge degree $n^2-1$, starting from $e_{1,1}$ and terminated at $e_{n,n}$. In other words $\mathcal{G}(\Pi_{n+1})$ has $2(\mathcal{D}_{n+1}-1)=2(2((n+1)-1)-1=2(2n-1)=4n-2$ Hamiltonian paths, each with edge degree $(n+1)^2-1$, starting from $e_{1,1}$ and terminating at $e_{n+1,n+1}$. And as our approach is based on construction, and based on the induction steps, and the fact that $\mathcal{G}(\Pi_n)\subset \mathcal{G}(\Pi_{n+1})$, hence clearly the relation of $\mathcal{G}(\Pi_{n+1})$-Cuntz-Krieger family will satisfy in the case of $S_i$'s associated with $\mathcal{G}(\Pi_{n+1})$.
\end{proof}
\begin{remark}
    If you have already noted that the 
     quantum state of our quantum systems will be $(\mathbb{C}^2)^{\otimes 2}=\mathbb{C}^2\otimes\mathbb{C}^2$, $(\mathbb{C}^2)^{\otimes 6}=(\mathbb{C}^2)^{\otimes 3}\otimes(\mathbb{C}^2)^{\otimes 3}$, $(\mathbb{C}^2)^{\otimes 10}=(\mathbb{C}^2)^{\otimes 5}\otimes (\mathbb{C}^2)^{\otimes 5}$, and so on, with dimensions $4, 64, 1024$ and so on, in order. And hence we might formulate the following claim.
\end{remark}
\begin{cl}
    The quantum system arising from $\mathcal{G}(\Pi_n)$ agrees and improves (in quantum scale) the ordinary classical systems!
\end{cl}
In other case, it is easy to see that $10$-qubit in the quantum system introduced in this paper will be equivalent to the 1 \textit{MB}=1024 \textit{Bytes} of the classical systems.
\section{Concluding remarks}
We believe that the research conducted in this paper is very interesting, and if we want to describe it in just one sentence it would be ``from simplicity to complexity''! 

\hspace*{0.2cm}Once again, and as in our previous works \cite{RH24,RH242}, we started from our toy example $\mathbb K[M_q(n)]$, and motivated by the graph $C^*$-algebras through Cuntz-Krieger $\mathcal{G}(\Pi_n)$-family for the finite locally connected directed graph $\mathcal{G}(\Pi_n)$, and then we introduced the first initial concept of a 2-qubit entangled quantum system based on $\mathcal{G}(\Pi_2)$, by discovering its relations with the Hamiltonian paths! 

\hspace*{0.2cm}Then we proposed a claim on if the general case also provides us with an entangled quantum system.

\hspace*{0.2cm}In continuation, for an interested reader, a direction which could be proposed is to look at the other characteristics (\ref{St:1}, (\ref{C:3}, and to prove that if they really are equivalent with (\ref{St:1}, and since this is a structural example, based on construction, hence we believe that it is almost applicable!

\section{Acknowledgement}
\hspace*{0.2cm} The author of this manuscript would like to express his sincere thanks and gratitude for the hospitality of the Department of Mathematics of the Azarbaijan Shahid Madani University, on which most of the work has been extracted and has been finalized, and he would like to acknowledge his postdoctoral grant with contract No. 117.d.22844 - 08.07.2023.

\hspace*{0.2cm} The research conducted in this research paper was also in part supported by a grant from IPM (No. 1403170014)

%\hspace*{0.2cm} The second author was supported by the Department of Mathematics of the Azarbaijan Shahid Madani University under grant No. 1402/270 - 19.04.2023.

%\hspace*{0.2cm}The authors of this manuscript have equal contributions to this work.

\def\cprime{$'$} \def\cprime{$'$} \def\cprime{$'$}
\providecommand{\bysame}{\leavevmode\hbox to3em{\hrulefill}\thinspace}
\providecommand{\MR}{\relax\ifhmode\unskip\space\fi MR }
\providecommand{\MRhref}[2]{%
	\href{http://www.ams.org/mathscinet-getitem?mr=#1}{#2}
}
\providecommand{\href}[2]{#2}

\let\itshape\upshape

\end{document}